\documentclass[12pt]{amsart}
\usepackage{a4wide}
\usepackage{amsthm,amsfonts,amsmath,mathrsfs,amssymb}
\usepackage{dsfont}
\usepackage[T1]{fontenc}
\usepackage[utf8]{inputenc}
\usepackage{fourier}
\def\a{\alpha}       \def\b{\beta}        \def\g{\gamma}
       \def\la{\lambda}     \def\om{\omega}
       \def\t{\theta}       \def\f{\phi}
                  
\def\ch{\chi}        

\def\e{\varepsilon}       

           \def\O{\Omega}
         \def\La{\Lambda}

\def\D{{\mathbb D}}     
\def\C{{\mathbb C}}     \def\N{{\mathbb N}}
     
\def\ca{{\mathcal A}}   \def\cb{{\mathcal B}}
   
\def\ch{{\mathcal H}}   
   \def\cm{{\mathcal M}}

\def\({\left(}       \def\){\right)}

\newtheorem{prop}{\sc Proposition}

\newtheorem{thm}[prop]{\sc Theorem}

\begin{document}
%%% Title
\title[Invertible and isometric WCO]{Invertible and isometric weighted composition operators}
%%% Information for the first author
\author[A. Mas]{Alejandro Mas}
\address{Departamento de Matem\'aticas, Universidad Aut\'onoma de
Madrid, 28049 Madrid, Spain}
\email{alejandro.mas@uam.es}
%%% Information for the second author
\author[D. Vukoti\'c]{Dragan Vukoti\'c}
\address{Departamento de Matem\'aticas, Universidad Aut\'onoma de
Madrid, 28049 Madrid, Spain} \email{dragan.vukotic@uam.es}
%\dedicatory{Dedicated to ... on the occasion of ...}
%%% General info
\subjclass[2010]{47B33, 30H10}
\keywords{Weighted composition operator; spaces of analytic functions; isometry; invertible operators.}
\date{08 May, 2021.}
%%%%%%%%%%%%%%%%%%%%%%%%%%%
%%% Abstract
%%%%%%%%%%%%%%%%%%%%%%%%%%%
\begin{abstract}
We consider abstract Banach spaces of analytic functions on general bounded domains that satisfy only a minimum number of axioms. We describe all invertible (equivalently, surjective) weighted composition operators acting on such spaces. For the spaces of analytic functions in the disk whose norm is given in terms of a natural seminorm (such as the typical spaces given in terms of derivatives), we describe all surjective weighted composition operators that are isometric. This generalizes a number of known results.
\end{abstract}

\maketitle

%%%%%%%%%%%%%%%%%%%%%%%%%%%%%%%%%%%%%%%%%%%%%%%%%%%%%%%%%%%%%%%%
\section{Introduction}
 \label{sec-intro}
%%%%%%%%%%%%%%%%%%%%%%%%%%%%%%%%%%%%%%%%%%%%%%%%%%%%%%%%%%%%%%%%

%%%%%%%%%%%%%%%%%%%%%%%%%%%%%%%%%%%%%%%%%%%%%%%%%%%%%%%%%%%%%%%%
\subsection{Weighted composition operators}
 \label{subsec-sbjct}
%%%%%%%%%%%%%%%%%%%%%%%%%%%%%%%%%%%%%%%%%%%%%%%%%%%%%%%%%%%%%%%%
For a function $F$ analytic in a planar domain $\Omega$ and an analytic map $\f$ of $\Omega$ into itself, the \textit{weighted composition operator\/} (often abbreviated as WCO) $W_{F,\f}$ with \textit{symbols\/} $F$ and $\f$ is defined formally by the formula $W_{F,\f}f = F (f\circ\f)$ as the composition followed by multiplication. Its boundedness depends on the space of analytic functions on which one considers it. Such operators are relevant for various reasons:
\begin{itemize}
 \item
They generalize both the pointwise multiplication operators (when $\f(z)\equiv z$) and the composition operators (when $F\equiv 1$);
 \item
The only surjective isometries of some classical spaces of analytic functions (Hardy, Bergman, quotient Bloch space) are precisely of this type (\textit{cf.\/} \cite{Fo1, Fo2}, \cite{Kol}, \cite{CW});
 \item
Some classical operators such as the Ces\`aro operator \cite{CS} and the Hilbert matrix operator \cite{DiS} can be written as average values of weighted composition operators;
 \item
Close connections have been found between the boundedness \cite{Shim} or compactness \cite{Sm} of certain weighted composition operators on weighted Bergman spaces and the well-known  Brennan conjecture on integrability of derivatives of conformal maps.
\end{itemize}
The earliest explicit references on WCOs are \cite{Ka} and \cite{Ki}. There is a large body of work on the subject: \cite{BDL}, \cite{M}, \cite{CHD1, CHD2}, \cite{OZ}, \cite{CZ}, \cite{AL}, \cite{CGP}, \cite{Gu}, \cite{GLW}, \cite{LMN}, \cite{BGJJ}, to mention only some more recent references on the boundedness, compactness, or questions related to the spectra of such operators.

%%%%%%%%%%%%%%%%%%%%%%%%%%%%%%%%%%%%%%%%%%%%%%%%%%%%%%%%%%%%%%%%
\subsection{The content of this note and related recent works}
 \label{subsec-cont}
%%%%%%%%%%%%%%%%%%%%%%%%%%%%%%%%%%%%%%%%%%%%%%%%%%%%%%%%%%%%%%%%
The first step towards understanding the spectrum of an operator consists in understanding its invertibility. A non-trivial WCO (one for which $\f\not\equiv\,$const and $F\not\equiv 0$) is injective, hence by an application of the closed graph (or open mapping) theorem, it is an invertible operator if and only if it is onto. The point of view of the present paper is to study the properties of surjective WCOs in a unified way on general Banach spaces of analytic functions that satisfy only a handful of axioms while still obtaining that same conclusions as in the known special situations, thus covering many cases in one stroke. Several recent papers have been  guided by a similar philosophy, for example, \cite{Bo}, \cite{CT}, \cite{HLNS}, \cite{AV}. Two general theorems regarding invertibility were proved in \cite{AV}, assuming different sets of five axioms that the space should satisfy. Theorem~\ref{thm-invert-WCO} of this note proves a similar result for
a much wider class of spaces. The key contribution of the present note with respect to \cite{AV} is that by changing essentially only one axiom we can include many more spaces than the ones covered before,  while at the same time the proof becomes simpler. 
\par
Isometries of various spaces defined in terms of derivatives were characterized in \cite{HJ}. This had been done earlier in \cite{CW} for the subspace $\{f\in\cb\,:\,f(0)=0\}$ of the Bloch space, where all isometries turn out to be WCO. However, in the full Bloch space $\cb$ there are more isometries so it is still a question of interest to characterize all isometric WCO on $\cb$. In the special case when $\,W_{F,\f}$ is a composition operator ($F\equiv 1$) the isometries were characterized in two different ways in \cite{Co} and \cite{MV}. The case of isometric multipliers  ($\f(z)=z$) is simpler and was covered in \cite{ADMV}. Theorem~\ref{thm-rotn} of the present note characterizes the surjective isometries among the WCOs on a rather general class of functional Banach spaces with a translation-invariant seminorm. The results include various cases that, to the best of our knowledge, were not covered before in the literature. The key point again consists in considering an axiom on which the earlier papers did not seem to focus.  

\medskip

%%%%%%%%%%%%%%%%%%%%%%%%%%%%%%%%%%%%%%%%%%%%%%%%%%%%%%%%%%%%%%%%
\section{Surjective WCOs on functional Banach spaces, revisited}
 \label{sec-fbs}
%%%%%%%%%%%%%%%%%%%%%%%%%%%%%%%%%%%%%%%%%%%%%%%%%%%%%%%%%%%%%%%%

%%%%%%%%%%%%%%%%%%%%%%%%%%%%%%%%%%%%%%%%%%%%%%%%%%%%%%%%%%%%%%%%
\subsection{A set of axioms}
 \label{subsec-axioms}
%%%%%%%%%%%%%%%%%%%%%%%%%%%%%%%%%%%%%%%%%%%%%%%%%%%%%%%%%%%%%%%%
Here we consider general Banach spaces of analytic functions on a bounded domain $\Omega\subset\C$. We shall write $\ch(\O)$ for the algebra of functions analytic in $\O$. Of course, the case of main interest is $\O=\D$, the unit disk.
\par
In the list below, only one axiom will be different with respect to \cite[Section~3]{AV}. However, this change is fundamental since the class of admissible spaces will now be much larger, yet at the same time the proofs will be notably shorter.
\par
We consider arbitrary Banach spaces $X\subset \ch(\Omega)$ that satisfy the following axioms:
\par
\begin{description}
\item[(A1)]
 All point evaluation functionals $\La_z$ are bounded on $X$.
\item[(A2)]
  $\mathbf{1}\in X$, where $\mathbf{1}(z)\equiv 1$.
\item[(A3)]
 Whenever $f\in X$, the function $\chi f$ is also in $X$, where $\chi(z)\equiv z$.
\item[(A4)]
 For every $u\in H^\infty(\Omega)$ that does not vanish in $\O\,$ and every $f\in X$, if $fu^n\in X$ for all $n\in\N=\{1,2,3,\ldots\}$, then $fu^\a\in X$ for some positive non-integer value $\a$.
\item[(A5)]
  Each automorphism of $\Omega$ induces a bounded composition operator in $X$.
\end{description}
\par\noindent
The first axiom is essentially equivalent to the requirement that the space be Banach as it shows that convergence in norm implies uniform convergence on compact subsets of $\,\Omega$. Also, this allows the use of the closed graph theorem to show that Axiom~\textbf{(A3)} implies that the shift operator is bounded on $X$.
\par
A function $u$ is said to be a \textit{pointwise multiplier\/} of $X$ if $u f\in X$ for all $f\in X$; we write and write $u\in \cm(X)$ to denote this. Thus, Axiom~\textbf{(A3)} can be rephrased by saying that the identity function $\chi\in\cm(X)$. Together with Axiom~\textbf{(A2)}, this shows that the space $X$ contains the polynomials.
\par\smallskip
Axiom~\textbf{(A4)} may not look so natural at first sight but can be explained as follows. Axiom~\textbf{(A1)} implies that each pointwise multiplier is a bounded function analytic in $\Omega$. However, the converse is false in many spaces. For example, for the Bloch space $\cb$ of the disk, it is well known that $u\in \cm(\cb)$ if and only if $u\in H^\infty$ and
\begin{equation}
 \sup_{z\in\D} (1-|z|^2) \log \frac{2}{1-|z|^2} |u^\prime (z)| <\infty\,.
 \label{eq-mult-cond}
\end{equation}
This was proved in several papers by different authors around 1990; see, for example, \cite{BS}. The situation is more complicated in the Dirichlet space where other conditions have to be added to the boundedness assumption on $u$; we refer the reader to \cite{St}. The spaces of the disk for which $\cm(X)=H^\infty$ have been characterized in \cite{AV} in terms of a domination property.
\par
In view of the above, assuming the condition $\cm(X) = H^\infty(\Omega)$ (that is: every bounded analytic function is a multiplier of $X$ into itself) would not provide a remedy as this would not cover the Bloch or the Dirichlet space. The next step would be to think of a weaker condition, assuming that every bounded non-vanishing function could be ``compressed'' in order to be made into a multiplier:
\begin{itemize}
 \item
For every $u\in H^\infty(\Omega)$ that does not vanish in the disk,  $fu^\a\in X$ for some non-integer value $\a>0$.
\end{itemize}
However, one can easily check that again this property is not satisfied in $\cb$. This condition can be weakened further:
\begin{itemize}
 \item
For every $u\in H^\infty(\Omega)$ that does not vanish in the disk and for every $f\in X$, if $fu\in X$ then $fu^\a\in X$ for some non-integer value $\a>0$.
\end{itemize}
It turns out that this property is satisfied in $\cb$ but it is not obvious how to verify it for the minimal analytic Besov space $B^1$. However, one can weaken the above requirement a little more, thus formulating our Axiom~\textbf{(A4)}. It turns out that it is fulfilled in most ``reasonable'' spaces of the disk.
\par\smallskip
%%%%%%%%%%%%%%%%%%%%%%%%%%%%%%%%%%%%%%%%%%%%%%%%%%%%%%%%%%%%%%%%
\subsection{Spaces that satisfy our axioms}
 \label{subsec-spaces}
%%%%%%%%%%%%%%%%%%%%%%%%%%%%%%%%%%%%%%%%%%%%%%%%%%%%%%%%%%%%%%%%
We review a partial list of spaces of analytic functions in the disk that satisfy all of the above axioms. The list is a bit larger than the list of the spaces satisfying the axioms from \cite{AV}.
\par\smallskip
$\bullet$ $H^\infty$, the (Hardy) space of all \textit{bounded analytic functions\/} in $\D$, equipped with the norm $\|f\|_\infty = \sup_{z\in\D}|f(z)|$, and the \textit{disk algebra\/} $\ca=H^\infty\cap C(\overline{\D})$, its subspace with the same norm.
\par\smallskip
$\bullet$ The standard \textit{Hardy spaces\/} $H^p$, $1\le p<\infty$, consisting of all $f\in\ch(\D)$ for which
$$
 \|f\|_{H^p} = \sup_{0\le r<1} \( \int_0^{2\pi} |f(r e^{i \t})|^p \frac{d\t}{2\pi}\)^{1/p} <\infty\,.
$$
\par
$\bullet$ The \textit{weighted Bergman spaces\/} $A^p_\a$, $1\le p<\infty$, $-1<\a<\infty$, of all $f\in\ch(\D)$ such that
$$
 \|f\|_{A^p_\a} = \( (\a+1) \int_\D |f(z)|^p (1-|z|^2)^\a dA(z)\)^{1/p} <\infty\,,
$$
where $dA$ denotes the normalized Lebesgue area measure on $\D$: $dA(z)=r dr d\t/\pi$, $z=re^{i\t}$.
\par\smallskip
$\bullet$
The general \text{mixed-norm spaces\/} $H(p,q,\a)$ of all functions $f$ in $\ch(\D)$ for which
$$
 \|f\|_{p,q,\a} = \( \a q \int_0^1 (1-r^2)^{\alpha q-1} M_p^q(r,f) 2 r dr\)^{1/q} <\infty\,, \quad 0<p\le\infty\,,\ 0<q<\infty\,, \ 0<\alpha<\infty\,,
$$
and
$$
 \|f\|_{p,\infty,\a} = \sup_{0\le r<1} (1-r^2)^{\alpha} M_p(r,f) <\infty\,.
$$
Note that $H(p,p,1/p)=A^p$.
\par\smallskip
$\bullet$ The (Korenblum-type) \textit{growth spaces\/}  $\mathcal{A}^{-\g}$, $\g>0$, of all $f\in\ch(\D)$ such that
$$
 \| f \|_{\mathcal{A}^{-\gamma}}=\sup_{z\in
\mathbb{D}}(1-|z|^{2})^{\gamma}|f(z)|.
$$
\par
$\bullet$ The \textit{weighted Bloch spaces\/} $\cb^\b$, $\b>0$, defined as the set of all $f\in\ch(\D)$ such that
$$
 \| f \|_{\mathcal{B}^{\beta}}=|f(0)|+\sup_{z\in \mathbb{D}} (1-|z|^{2})^{\beta}|f'(z)|<\infty.
$$
Of course, this scale includes the standard \textit{Bloch space\/}, obtained for the value $\b=1$.
\par\smallskip
$\bullet$ The \textit{logarithmic Bloch spaces\/}  $\mathcal{B}_{\log^{\gamma}}$, $\gamma\in \mathbb{R}$, where
$$
\Vert f \Vert_{\mathcal{B}_{\log^{\gamma}}}=|f(0)|+\sup_{z\in \mathbb{D}} (1-|z|^{2})\log^{\gamma}\left( \frac{2}{1-|z|^{2}}\right)|f'(z)|<\infty.
$$
(For $\g=1$, the reader will recognize the familiar condition \eqref{eq-mult-cond}, which is part of the motivating for studying these spaces while the value $\g=0$ gives the standard Bloch space.)
\par\smallskip
$\bullet$ The space $BMOA$ of analytic functions of \textit{bounded mean oscillation\/}, defined by
$$
\Vert f \Vert_{BMOA}=|f(0)|+\sup_{a\in \mathbb{D}} \Vert f \circ \varphi_{a}-f(a)\Vert_{H^{2}}<\infty,
$$
where $\varphi_a(z)=\frac{a-z}{1-\overline{a}z}$ is the standard disk automorphism which is an involution. It is convenient to use the following equivalent norm in $BMOA$ (\textit{cf\/.} \cite[Theorem 6.2]{Gi}):
$$
 \Vert f \Vert_{\star}=|f(0)|+\sup_{a\in \mathbb{D}}\left( \int_{\mathbb{D}} |f'(z)|^{2}(1-|\varphi_{a}(z)|^{2})\, d A(z)\right)^{\frac{1}{2}}\,.
$$
\par
$\bullet$ The \textit{weighted Besov (or Dirichlet-type) spaces\/}  $B^{p,\a}$ with $1\le p <\infty$ and $-1<\a<\infty$ of all $f\in \ch(\D)$ for which
$$
\Vert f \Vert_{B^{p,\a}}=|f(0)|+\Vert f' \Vert_{A^{p}_{\a}}<\infty\,.
$$
This includes the conformally invariant \textit{analytic (diagonal) Besov spaces\/} $B^{p,p-2}$, $1<p<\infty$; a further special case $p=2$ yields the \textit{Dirichlet space\/}.
\par\smallskip
$\bullet$ The \textit{minimal Besov space\/} $B^1$ which can be defined in terms of atomic decomposition (infinite sums of disk automorphisms with $\ell^1$ coefficients) but is more conveniently seen as the space of all analytic functions in $\D$ such that $f^{\prime\prime}\in A^1$, equipped with the norm
$$
 \Vert f \Vert_{B^{1}}=|f(0)|+|f'(0)|+\int_{\mathbb{D}} |f^{\prime\prime}(z)|\, dA(z) \,.
$$
\par\medskip
It would take us afar to display all the technicalities in checking that the spaces from the above list fulfill the five axioms. However, it seems convenient to give several indications.
\par\smallskip
For all of the spaces listed, Axiom~\textbf{(A1)} is satisfied in view of the appropriate and well-known pointwise estimates (actually, they are required in order to show that the space in question is complete). For Hardy and Bergman spaces, as well as for the Bloch space, we refer the reader to \cite{D, DuS, HKZ} and for analytic Besov spaces, to \cite{HW}. Just to illustrate some arguments, for $f\in \mathcal{B}^\b$ and $z=re^{it} \in \mathbb{D}$, the standard pointwise estimate:
\begin{align*}
\vert f(z)-f(0)\vert & \leq \int_{0}^{r} \left| f'(\rho e^{it})  \right| \, d\rho  \leq \Vert f \Vert_{\mathcal{B}^{\beta}} \int_{0}^{r} \frac{1}{(1-\rho^{2})^{\beta}} \, d \rho
\leq \left\{
 \begin{array}{ll}
 C \Vert f \Vert_{\mathcal{B}^{\beta}} & \mathrm{if } \, \beta<1, \\ [2mm]
 \frac{1}{2}\log \left( \frac{1+r}{1-r}\right) \Vert f \Vert_{\mathcal{B}^{\beta}} & \mathrm{if } \,  \beta=1,  \\ [2mm]
 \frac{C}{(1-r^2)^{\beta-1}} \Vert f \Vert_{\mathcal{B}^{\beta}}& \mathrm{if } \,  \beta>1,
 \end{array}
 \right.
\end{align*}
shows that Axiom~\textbf{(A1)} is satisfied.
\par
Since $BMOA$ and analytic Besov spaces are contained in $\cb$ and the inclusion is continuous, the above estimate for $\b=1$ can be used.
\par
For the logarithmic Bloch space, one can actually produce a unified estimate, because an  associated integral is convergent independently of $\g$:
$$
 \vert f(z)-f(0)\vert \leq \Vert f \Vert_{\mathcal{B}_{\log^{\gamma}}} \int_{0}^{r} \frac{1}{(1-\rho^{2})\log^{\gamma}\left( \frac{2}{1-\rho^{2}}\right)} \, d \rho \lesssim \frac{1}{(1-r^{2})^{\varepsilon}}\,,
$$
for $f\in \mathcal{B}_{\log^{\gamma}}$, $z=re^{it} \in \mathbb{D}$, and $\e>0$.
\par\smallskip
Axiom~\textbf{(A2)} holds trivially in all spaces from our list.
\par\smallskip
Axiom~\textbf{(A3)} is trivially verified in Hardy, weighted Bergman, and growth spaces. In Bloch-type and Besov-type spaces, essentially, one has to use the fact that a function has better properties than its derivative (in terms of boundedness and integrability). In order to estimate $|(\chi f)^\prime(z)|=|f(z)+zf^\prime(z)|$, it then suffices to add up the obvious estimates for the derivative and the function.
\par
In $BMOA$, we can argue as follows:
\begin{align*}
 \Vert \chi f \Vert_{\star}^{2} & =\sup_{a\in \mathbb{D}}\int_{\mathbb{D}} |zf'(z)+f(z)|^{2}(1-|\varphi_{a}(z)|^{2})\, d A(z)
\\
 &\lesssim \sup_{a\in \mathbb{D}}\int_{\mathbb{D}} |f'(z)|^{2}(1-|\varphi_{a}(z)|^{2})\, d A(z) + \sup_{a\in \mathbb{D}}\int_{\mathbb{D}} |f(z)|^{2}(1-|\varphi_{a}(z)|^{2})\, d A(z)
\\
 &\lesssim \left( 1+ \sup_{a\in \mathbb{D}}\int_{\mathbb{D}} \left( 1+\frac{1}{2} \log \frac{1+|z|}{1-|z|} \right)^{2}(1-|\varphi_{a}(z)|^{2})\, d A(z) \right) \Vert f\Vert_{\star}^{2}
\\
& \leq \left( 1+\int_{\mathbb{D}} \left( 1+\frac{1}{2} \log \frac{1+|z|}{1-|z|} \right)^{2}\, d A(z)  \right) \Vert f\Vert_{\star}^{2}\,,
\end{align*}
where the integral in the last line is convergent (as is usual, $\lesssim$ means that the quantity on the left can be bounded by the quantity on the right times a fixed constant).
\par
In $B^1$, it is convenient to use the Littlewood-Paley formula to check that the axiom is satisfied.
\par\smallskip
The main issue is, of course, checking our Axiom~\textbf{(A4)}. This is quite clear for the disk algebra, Hardy spaces, weighted Bergman spaces or growth spaces.
\par
Given $f \in \mathcal{B}^{\beta}$ and a non-vanishing function  $u\in H^{\infty}$ such that   $fu^{n}\in  \mathcal{B}^{\beta}$ for all $n\in \mathbb{N}$, if $\alpha >1$ then
$$
 \Vert fu^{\alpha} \Vert_{\mathcal{B}^{\beta}} = \vert f(0)u^{\alpha}(0) \vert+\sup_{z\in \mathbb{D}} (1-|z|^{2})^{\beta}\vert  f'(z)u^{\alpha}(z)+\alpha u^{\alpha-1}(z)u'(z)f(z)\vert.
$$
Since $\,\vert f(0)u^{\alpha}(0) \vert \le \Vert f\Vert_{\mathcal{B}^{\beta}} \Vert u \Vert_{H^{\infty}}^{\alpha}
$
and
\begin{eqnarray*}
 \vert f'(z)u^{\alpha}(z)+\alpha u^{\alpha-1}(z)u'(z)f(z)\vert &=& \vert (1-\alpha) f'(z)u^{\alpha}(z)+\alpha u^{\alpha-1}(z)\left( f'(z)u(z)+u'(z)f(z)\right)\vert
\\
 & \le & (\alpha-1) \vert   f'(z)u^{\alpha}(z)\vert +\alpha \Vert u\Vert^{\alpha-1}_{H^{\infty}} \vert (fu)'(z)\vert,
\end{eqnarray*}
we have
\begin{align*}
\Vert fu^{\alpha} \Vert_{\mathcal{B}^{\beta}}&\leq \Vert u \Vert_{H^{\infty}}^{\alpha}\Vert f \Vert_{\mathcal{B}^{\beta}}+\alpha  \Vert u\Vert^{\alpha-1}_{H^{\infty}}\Vert fu \Vert_{\mathcal{B}^{\beta}} +(\alpha-1)  \Vert u\Vert^{\alpha}_{H^{\infty}}\Vert f \Vert_{\mathcal{B}^{\beta}}.
\end{align*}
A similar argument works for logarithmic Bloch spaces and $BMOA$.
\par
It is the case of $B^1$ that requires most work. It also explains why we may need values $n>1$ in Axiom~\textbf{(A4)}. Given   $f \in B^{1}$ and a non-vanishing function  $u\in H^{\infty}$ such that   $fu^{n}\in  B^{1}$ for all $n\in \mathbb{N}$, if $\alpha >2$ then
\begin{eqnarray*}
 |(u^\alpha f)^{\prime\prime}(z)| &=& |\alpha (\alpha-1) u^{\alpha-2}(z)(u'(z))^{2}f(z)+\alpha u^{\alpha-1}(z)u''(z)f(z)
\\
 & & +2\alpha u^{\alpha-1}(z) u'(z)f'(z)+u^\alpha(z)f^{\prime\prime}(z)|.
\end{eqnarray*}
Therefore
\begin{align*}
 |(u^\alpha f)''(z)|\leq  & \frac{\alpha(\alpha-1)}{2} \Vert u \Vert_{H^{\infty}}^{\alpha-2}\left| (fu^{2})''(z)-2\left( 1-\frac{1}{\alpha-1}\right)u(z)u''(z)f(z)\right.\\
 &-4\left( 1-\frac{1}{\alpha-1}\right) u(z)u'(z)f'(z) \left.-\left( 1-\frac{2}{\alpha(\alpha-1)}\right) u^2(z)f''(z)\right|.
\end{align*}
Hence,
\begin{align*}
 |(u^\alpha f)''(z)|\lesssim    \Vert u \Vert_{H^{\infty}}^{\alpha-2} | (fu^{2})''(z)| +\Vert u \Vert_{H^{\infty}}^{\alpha-2} \left| g(z) \right|,
 \end{align*}
 where
 \begin{align*}
 g(z)=-2\left( 1-\frac{1}{\alpha-1} \right)&u(z)u''(z)f(z)-4\left( 1-\frac{1}{\alpha-1}\right) u(z)u'(z)f'(z) \\
 &-\left( 1-\frac{2}{\alpha(\alpha-1)}\right) u^2(z)f''(z).
 \end{align*}
A direct computation yields
  \begin{align*}
 |g(z)|&\lesssim \Vert u \Vert_{H^{\infty}}|u''(z)f(z)+2u'(z)f'(z)+u(z)f''(z)|+ \Vert u \Vert_{H^{\infty}}^{2}|f''(z)|\\
 &= \Vert u \Vert_{H^{\infty}}|(uf)''(z)|+ \Vert u \Vert_{H^{\infty}}^{2}|f''(z)|.
 \end{align*}
Therefore
  \begin{align*}
 \int_{\mathbb{D}} |(u^\alpha f)''(z)|\, dA(z) \lesssim \Vert u \Vert_{H^{\infty}}^{\alpha-2} \Vert fu^{2}\Vert_{B^{1}} + \Vert u \Vert_{H^{\infty}}^{\alpha-1} \Vert fu\Vert_{B^{1}}+\Vert u \Vert_{H^{\infty}}^{\alpha}\Vert f\Vert_{B^{1}}.
 \end{align*}
\par\smallskip
In Hardy and weighted Bergman spaces, Axiom~\textbf{(A5)} relies on the Littlewood subordination principle. In the growth spaces and the Bloch space, this follows from the Schwarz-Pick lemma. In weighted Besov spaces one also has to use the fact that the derivative of a fixed disk automorphism is bounded from above and bounded away from zero. The Bloch space, $BMOA$, $B^1$, and the disk algebra are easily seen to be conformally invariant (in the sense that if we choose an appropriate equivalent norm in the space, then compositions with self-maps of the disk do not increase the corresponding seminorm).

%%%%%%%%%%%%%%%%%%%%%%%%%%%%%%%%%%%%%%%%%%%%%%%%%%%%%%%%%%%%%%%%
\subsection{Invertibility (ontoness) of WCOs}
 \label{subsec-invert}
%%%%%%%%%%%%%%%%%%%%%%%%%%%%%%%%%%%%%%%%%%%%%%%%%%%%%%%%%%%%%%%%
We now prove our first main result. It should be noted that the requirement that $\a$ be a non-integer in Axiom~\textbf{(A4)} is relevant in the proof (in order to produce a non-analytic function and thus obtain a contradiction). As we have seen, the results obtained previously such as \cite[Theorem~5 and Theorem~7]{AV} do not cover nearly as many spaces as the theorem below.
%%%%%%%%%%
\begin{thm} \label{thm-invert-WCO}
Let $X\subset\ch(\O)$ be any functional Banach space on a bounded planar domain $\,\O$ in which the axioms \textbf{(A1) - (A4)} are satisfied, and suppose that a weighted composition operator $W_{F,\f}$ is bounded in $X$.
\par
(a) If $\,W_{F,\f}$ is invertible in $X$ (equivalently, onto), then its composition symbol $\f$ is an automorphism of \,$\O$, the multiplication symbol $\,F$ does not vanish in \,$\O$, and the inverse operator $W_{F,\f}^{-1}$ is another weighted composition operator $W_{G,\psi}$, whose symbols are:
\begin{equation}
 G=\frac{1}{F\circ \f^{-1}}, \qquad \psi=\f^{-1}\,.
 \label{inv-symb}
\end{equation}
\par
(b) Assuming that all Axioms~\textbf{(A1)-(A5)} hold, the converse is true as well, so we have the following characterization.
\par
The weighted composition operator $W_{F,\f}$ is invertible on $X$ if and only if its composition symbol $\f$ is an automorphism of \,$\O$, the multiplication symbol $\,F$ does not vanish in \,$\O$, and $1/F\in\cm (X)$.
\par
If this is the case, then $F$ is also a self-multiplier of $X$ and the inverse operator is $W_{G,\psi}$ whose symbols are given by \eqref{inv-symb}.
\end{thm}
%%%%%%%%%%
\begin{proof} We first prove part (a), using only axioms~\textbf{(A1)-(A4)}.
\par
Since $\textbf{1}\in X$ by Axiom~\textbf{(A2)} and $\,W_{F,\f}$ is onto by assumption, $F (f\circ\f) = \mathbf{1}$ must hold for some $f\in X$, hence $F$ cannot vanish in $\Omega$. We now show that $\f$ is an automorphism of $\O$.
\par
Since $F=W_{F,\f}\textbf{1}\in X$, we know that $\chi F \in X$ by Axiom~\textbf{(A3)}. Since the operator $\,W_{F,\f}$ is onto, there exists a function $f\in X$ such that $\chi F = F (f\circ\f)$. It follows that $f\circ\f \equiv \chi$ and from here it is immediate that $\f$ is univalent: if $\f(a) = \f(b)$ then $a = f(\f(a)) = f(\f(b)) = b$.
\par
We next show that $\f(\Omega)=\Omega$. Suppose that, on the contrary, $\f$ omits some value $\om\in \Omega$. Then the function $\f-\om$ is bounded and does not vanish in $\O$. (This is where we use the assumption about boundedness of $\O$.) Note that by Axiom~\textbf{(A2)} and Axiom~\textbf{(A3)}, the function $\chi^k\in X$, hence
$$
 F\in X, \qquad F (\f-\om)^n = \sum_{k=0}^n \binom{n}{k}\om^{n-k} W_{F,\f}(\chi^k) \in X, \qquad n\in\N\,.
$$
Hence, by Axiom~\textbf{(A4)} there exists a non-integer value $\a>0$ such that $F (\f-\om)^\a\in X$.
\par
The operator $\,W_{F,\f}$ is onto, hence for some $f\in X$ we have $F (f\circ\f)=F (\f-\om)^\a$. But $\f$ is univalent, hence not constant. Therefore the equality  $f(z)= (z-\om)^\a$ holds on the non-empty open set $\f(\Omega)$, hence in all of $\O$. However, the function on the right is not analytic in $\O$, and therefore $f\not\in X$, which is absurd.
\par
We have, thus, shown that $\f$ is an automorphism of $\Omega$.
\par
Given $g\in X$, from the assumption that the operator is invertible and solving the equation $F (f\circ\f)=g$ for $f$ yield by straightforward computation that formulas \eqref{inv-symb} hold.
\par\medskip
(b) This part is similar to the proof given in \cite{AV} but we  simplify the arguments given there.
\par\smallskip
We first prove the forward implication, starting from the assumption that $W_{F,\f}$ is invertible. By the first part of the theorem, we know that $\f$ is an automorphism of $\Omega$.
\par
To see that $1/F\in\cm(X)$, let $g\in X$ be arbitrary. Since $\,W_{F,\f}$ is onto, there exists a function $f\in X$ such that $\,F (f\circ \f)=g$. But $f\circ\f \in X$ by Axiom~\textbf{(A5)}, hence $g/F$ is analytic in $\Omega$ and $g/F\in X$.
\par\medskip
Now for the reverse implication. Assuming that all five axioms are satisfied and that $\f$ is an automorphism of $\O$, $F$ does not vanish, and $1/F\in\cm(X)$, we argue as follows. Since $\f$ is an automorphism of $\Omega$, so is its inverse function $\f^{-1}$. Given an arbitrary function $f\in X$, we also have $f/F\in X$ and therefore also $(f\circ\f^{-1})/(F\circ\f^{-1})\in X$. Thus, the operator $W_{G,\psi}$ where
$$
 G=\frac{1}{F\circ\f^{-1}}, \qquad \psi=\f^{-1}
$$
maps $X$ into itself. By an application of the closed graph theorem, which is possible thanks to Axiom~\textbf{(A1)} and the principle of uniform boundedness, $W_{G,\psi}$ is a bounded operator on $X$. Now one easily checks directly that
$$
 W_{G,\psi} W_{F,\f} f = W_{F,\f} W_{G,\psi} f
$$
for all $f\in X$, hence the operator $W_{F,\f}$ has bounded inverse $W_{G,\psi}$.
\end{proof}
%%%%%%%%%%

\medskip

%%%%%%%%%%%%%%%%%%%%%%%%%%%%%%%%%%%%%%%%%%%%%%%%%%%%%%%%%%%%%%%%
\section{Surjective isometries among WCOs on spaces with translation-invariant seminorm}
 \label{sec-isom}
%%%%%%%%%%%%%%%%%%%%%%%%%%%%%%%%%%%%%%%%%%%%%%%%%%%%%%%%%%%%%%%%
Our second result proved below focuses on the spaces of Bloch or Besov type. We show that in such spaces the only surjective linear isometries among the WCOs are the most obvious ones. This was not explicitly stated, even for the Bloch space and the usual composition operators, in \cite{Co} or \cite{MV}, though in this particular case it can be deduced from the results obtained there after some discussion.
\par
Specifically, we consider function spaces $X$ of the disk that, in addition to Axioms~\textbf{(A1)-(A4)}, also satisfy the following condition:
\par
\begin{description}
\item[A6]
 The norm in $X$ is given by the formula $\|f\|=|f(0)|+p(f)$, where $p(f)$ is a seminorm which is translation-invariant: $p(f+C)=p(f)$, for all $f\in X$ and all $C \in\C$.
\end{description}
\par\noindent
It is trivial but important to note that if $X$ has this property, then $p(f)=0$ implies that $f$ is a constant function. Indeed, the function $f-f(0)$ vanishes at the origin, hence in view of the norm formula and translation invariance we have
$$
 \|f-f(0)\| = p(f-f(0)) = p(f) = 0
$$
which implies that $f=f(0)$, a constant function.
\par
One easily verifies directly that the spaces such as $\cb^\b$, $\cb_{\log^\g}$, $B^{p,\a}$, or $BMOA$ also satisfy Axiom~\textbf{(A6)}, in addition to the earlier ones.
\par
%%%%%%%%%%
\begin{thm} \label{thm-rotn}
Let $X\subset\ch(\D)$ be a functional Banach space of the disk that satisfies Axioms~\textbf{(A1) - (A4)} and Axiom~\textbf{(A6)}. If $\,W_{F,\f}$ is a surjective isometric WCO on $X$, then $F$ is a unimodular constant and $\f$ is a rotation.
\end{thm}
%%%%%%%%%%
\begin{proof}
By assumption, our operator is onto, hence there exists $f_0\in X$ such that $F(f_0\circ\f)\equiv \mathbf{1}$. Therefore $F$ does not vanish in $\D$.
\par
In order to show that $\f(0)=0$, consider the function $f_\la$ given by $f_\la(z)=1+\la z$, with $|\la|=1$. By axioms \textbf{(A2)} and \textbf{(A3)}, $f_\la\in X$ and, by Axiom~\textbf{(A6)}, the  homogeneity of the seminorm and the assumption that $W_{F,\f}$ is an isometry, we have
$$
 1+\|z\| = 1 + p(z) = 1 + p(f_\la) =\|f_\la\| = \|F+\la F \f\| \le \|F\| + \|F \f\| = \|\mathbf{1}\| + \|z\| = 1+\|z\|\,.
$$
Thus, equality must hold throughout, meaning that $\|F+\la F \f\| = \|F\| + \|F \f\|$. Hence
$$
 |(F+\la F \f)(0)| + p(F+\la F \f) =  |F(0)| + p(F) + |(F \f)(0)| + p(F \f)\,.
$$
Since $p(F+\la F \f) \le p(F) + p(F \f)$, it follows that
$$
 |(F+\la F \f)(0)| \ge  |F(0)| + |(F \f)(0)|\,.
$$
We know that $F(0)\neq 0$, hence
$$
 |1+\la \f(0)| \ge  1 + |\f(0)|\,,
$$
for all $\la$ with $|\la|=1$, which is only possible when $\f(0)=0$.
On the other hand, by part (a) of Theorem~\ref{thm-invert-WCO} (using only Axioms~\textbf{(A1) - (A4)}), $\f$ must be a disk automorphism. Hence, it is a rotation.
\par
Now, recalling that the function $f_0\in X$ chosen earlier has the property $F(f_0\circ\f)\equiv \mathbf{1}$, we obtain $F(0) f_0(0)=1$. On the other hand, the same property implies that
$$
 |f_0(0)| \le \|f_0\| = \|W_{F,\f}(f_0)\| = \|\mathbf{1}\|=1 \,.
$$
Also, $|F(0)|\le\|F\|=\|W_{F,\f}(\mathbf{1})\|=\|\mathbf{1}\|=1$, hence we must have $|F(0)|=|f_0(0)|=1$. Therefore $p(F)=0$, hence $F$ is a constant of modulus one. This proves the statement.
\end{proof}
%%%%%%%%%%
\par 
To conclude, we would like to observe that it does not seem likely that one could reduce any further the number of axioms needed to obtain the conclusions of both results of this note.
\par\medskip
\textsc{Acknowledgments}. {\small The authors were supported by PID2019-106870GB-I00, Spain.

%%%%%%%%%%%%%%%%%%%%%%%%%%%%%%%%%%%%%%%%%%%%%%%%%%%%%%%%%%%%%%%%


\begin{thebibliography}{99}
%%%%%%%%%%%%%%%%%%%%%%%%%%%%%%%%%%%%%%%%%%%%%%%%%%%%%%%%%%%%%%%%

\bibitem{ADMV}
A. Aleman, P.L. Duren, M.J. Mart\'{\i}n, D. Vukoti\'c, Multiplicative isometries and isometric zero-divisors of function spaces in the disk, \emph{Canad. J. Math.\/} \textbf{62} (2010), no. 5, 961--974.

\bibitem{AV}
I. Arévalo, D. Vukoti\'c, Weighted compositions operators in functional Banach spaces: an axiomatic approach, \textit{J. Spectral Theory\/} \textbf{2} (2020), no. 2, 673--701.

\bibitem{AL}
R. Aron, M. Lindstr\"om, Spectra of weighted composition operators on weighted Banach spaces of analytic functions, \textit{Isr. J. Math.\/} \textbf{141} (2004), 263–-276. 

\bibitem{BDL}
J. Bonet, P. Domanski, M. Lindstr\"om, Essential norm and weak compactness of composition operators on weighted Banach spaces of analytic functions, \textit{Canad. Math. Bull\/}. \textbf{42} (1999), no. 2, 139--148.

\bibitem{BGJJ}
J. Bonet, M.C. Gómez-Collado, E. Jordá, D. Jornet, Nuclear weighted composition operators on weighted Banach spaces of analytic functions, \textit{Proc. Amer. Math. Soc.\/} \textbf{149}  (2021), no. 1, 311–-321.

\bibitem{Bo}
P.S. Bourdon, Invertible weighted composition operators, \emph{Proc. Amer. Math. Soc.\/}, \textbf{142} (2014), no. 1, 289--299.

\bibitem{BS}
L. Brown, A.L. Shields, Multipliers and cyclic vectors in the Bloch space, \textit{Michigan Math. J.} \textbf{38} (1991), 141--146.

\bibitem{CGP}
I. Chalendar, E. Gallardo-Guti\'errez, J.R. Partington, Weighted composition operators on the Dirichlet space: boundedness and spectral properties, \textit{Math. Ann.\/} \textbf{363} (2015), no. 3-4, 1265–-1279.

\bibitem{CS}
J.A. Cima, A.G. Siskakis, Cauchy transforms and Cesàro averaging operators, \textit{Acta Sci. Math. (Szeged)\/} \textbf{65}  (1999), no. 3-4, 505–-513.

\bibitem{CW}
J.A. Cima, W.R. Wogen, On isometries of the Bloch space, \textit{Illinois J. Math\/.} \textbf{24} (1980), no. 2, 313--316.

\bibitem{Co}
F. Colonna, Characterization of the isometric composition operators on the Bloch space, \textit{Bull. Austral. Math. Soc.} \textbf{72} (2005), 283--290.

\bibitem{CT}
F. Colonna, M. Tjani, Operator norms and essential norms of weighted composition operators between Banach spaces of analytic functions,
\textit{J. Math. Anal. Appl.} \textbf{434} (2016), No. 1, 93-–124.

\bibitem{CHD1}
M. Contreras and A. G. Hern\'andez-D\'{\i}az, Weighted composition operators on weighted Banach spaces of analytic functions, \textit{J. Austral. Math. Soc. (Series~A)} \textbf{69} (2000), 41--60.

\bibitem{CHD2}
M.D. Contreras, A.G. Hern\'andez-D\'{\i}az, Weighted composition operators on Hardy spaces, \textit{J. Math. Anal. Appl.} \textbf{263}  (2001), no. 1, 224--233.

\bibitem{CZ}
Z.~\v Cu\v ckovi\'c and R.~Zhao, Weighted composition operators on the Bergman space, \textit{J. London Math. Soc. (2)} \textbf{70} (2004), no. 2, 499–-511.

\bibitem{DiS}
E. Diamantopoulos, A. G. Siskakis, Composition operators and the
Hilbert matrix, \textit{Studia Math.\/} \textbf{140} (2000), no. 2,
191--198.

\bibitem{D}
P.L.~Duren, \textit{Theory of $H^p$ Spaces\/}, Pure and Applied
Mathematics, Vol.~\textbf{38}, Second edition, Dover, Mineola, New
York 2000.

\bibitem{DuS}
P.L.~Duren and A.P.~Schuster, \emph{Bergman Spaces\/}, Graduate
Studies in Mathematics, American Mathematical Society, Providence,
RI, 2004.

\bibitem{Fo1}
F. Forelli, The isometries of $H^p$, \textit{Canad. J. Math.\/} \textbf{16} (1964), 721–-728.

\bibitem{Fo2}
F. Forelli, A theorem on isometries and the application of it to the isometries of $H^p(S)$ for $2<p<\infty$, \textit{Canad. J. Math.\/} \textbf{25} (1973), 284–-289.

\bibitem{GLW} 
P. Galindo, M. Lindstr\"om, N. Wikman, Spectra of weighted composition operators on analytic function spaces, \textit{Mediterr. J. Math.\/} \textbf{17} (2020), no. 1, paper no. 34, 22 pp.

\bibitem{Gi}
D. Girela, Analytic functions of bounded mean oscillation. In: Complex function spaces (Mekrij\"arvi, 1999), 61–170, \textit{Univ. Joensuu Dept. Math. Rep. Ser.\/} \textbf{4}, Univ. Joensuu, Joensuu 2001.

\bibitem{Gu}
G.~Gunatillake, Invertible weighted composition operators, \textit{J. Funct. Anal.\/} \textbf{261} (2011) 831–-860.

\bibitem{HKZ}
H.~Hedenmalm, B.~Korenblum, and K.~Zhu, \textit{Theory of Bergman
Spaces\/}, Graduate Texts in Mathematics \textbf{199}, Springer,
New York, Berlin, etc. 2000.

\bibitem{HW}
F. Holland, D. Walsh, Growth estimates for functions in the Besov spaces $A_p$, \textit{Proc. Roy. Irish Acad. Sect. A\/} \textbf{88} (1988), 1–-18.

\bibitem{HJ}
W. Hornor, J.E. Jamison, Isometries of some Banach spaces of analytic functions, \textit{Integral Equations Operator Theory\/} \textbf{41} (2001), no. 4, 410--425.

\bibitem{HLNS}
O.~Hyvärinen, M.~Lindström, I.~Nieminen, E.~Saukko, Spectra of weighted composition operators with automorphic symbols, \textit{J. Funct. Anal.\/} \textbf{265} (2013), no. 8, 1749–-1777.

\bibitem{Ka}
H. Kamowitz, Compact operators of the form $u C_{\f}$, \textit{Pacific J. Math.} \textbf{80} (1979), no. 1, 205--211.

\bibitem{Ki}
A.K.~Kitover, The spectrum of operators in ideal spaces (Russian), Investigations on linear operators and the theory of functions, VII. \textit{Zap. Nau\v cn. Sem. Leningrad. Otdel Mat. Inst. Steklov (LOMI)\/} \textbf{65} (1976), 196-–198, 209–-210.

\bibitem{Kol}
C.J. Kolaski, Isometries of Bergman spaces over bounded Runge domains, \textit{Canad. J. Math.\/} \textbf{33} (1981), no. 5, 1157-–1164.

\bibitem{LMN}
M. Lindstr\"om, S. Miihkinen, P.J. Nieminen, Rigidity of weighted composition operators on $H^p$, \textit{Ann. Acad. Sci. Fenn. Math.\/} \textbf{45} (2020), no. 2, 825–-828.

\bibitem{MV}
M.J. Mart\'{\i}n, D. Vukoti\'c, Isometries of the Bloch space among the composition operators, \textit{Bull. London Math. Soc.} \textbf{39}  (2007),  151--155.

\bibitem{M}
A. Montes-Rodr\'{\i}guez, Weighted composition operators on weighted Banach spaces of analytic functions, \textit{J. London Math. Soc. (2)\/} \textbf{61} (2000), no. 3, 872–-884.

\bibitem{OZ}
S. Ohno and R. Zhao, Weighted composition operators on the Bloch space, \textit{Bull. Austral. Math. Soc.} \textbf{63} (2001), 177--185.

\bibitem{Shim}
S. Shimorin, Weighted composition operators associated with conformal mappings, In: Quadrature domains and their applications, 217–237, \textit{Oper. Theory Adv. Appl.} \textbf{156}, Birkh\"auser, Basel 2005.

\bibitem{Sm}
W. Smith, Brennan's conjecture for weighted composition operators, In: Recent advances in operator-related function theory, 209–214, \textit{Contemp. Math.} \textbf{393}, Amer. Math. Soc., Providence, RI 2006.

\bibitem{St}
D. Stegenga, Multipliers of the Dirichlet space, \textit{Illinois J. Math.} \textbf{24} (1980), No. 1, 113--139.
\end{thebibliography}
\end{document}